\documentclass[12pt,reqno]{amsart}
\usepackage{amssymb}
\usepackage[all]{xy}
\usepackage{hyperref}
\usepackage{enumitem}

\theoremstyle{plain}
\newtheorem{prop}{Proposition}
\newtheorem{thm}[prop]{Theorem}
\newtheorem{cor}[prop]{Corollary}
\newtheorem{lem}[prop]{Lemma}

\theoremstyle{definition}
\newtheorem{defi}[prop]{Definition}

\theoremstyle{remark}

\newtheorem{rem}[prop]{Remark}

\numberwithin{prop}{section}
\numberwithin{ques}{section}
\numberwithin{equation}{section}

\DeclareMathOperator{\Subgr}{Subgr}

\DeclareMathOperator{\Res}{Res}
\DeclareMathOperator{\Cor}{Cor}

\DeclareMathOperator{\Ker}{Ker}

\DeclareMathOperator{\id}{id}

\newcommand{\Bhat}{\hat{B}}

\newcommand{\alphahat}{\hat{\alpha}}

\newcommand{\ca}[1]{\mathcal{#1}}

\newcommand{\F}{\mathbb{F}}

\newcommand{\bfB}{\mathbf{B}}
\newcommand{\bfA}{\mathbf{A}}
\newcommand{\bfG}{\mathbf{G}}

\newcommand{\suchthat}{\mathop{|\,}}

\newcommand{\argu}{\hbox to 7truept{\hrulefill}}

\newcommand{\G}{\ca{G}}

\newcommand{\C}{\ca{C}}

\newcommand{\normal}{\triangleleft}

\makeatletter
\def\moverlay{\mathpalette\mov@rlay}
\def\mov@rlay#1#2{\leavevmode\vtop{%
   \baselineskip\z@skip \lineskiplimit-\maxdimen
   \ialign{\hfil$\m@th#1##$\hfil\cr#2\crcr}}}
\newcommand{\charfusion}[3][\mathord]{
    #1{\ifx#1\mathop\vphantom{#2}\fi
        \mathpalette\mov@rlay{#2\cr#3}
      }
    \ifx#1\mathop\expandafter\displaylimits\fi}
\makeatother

\newcommand{\bigdotcup}{\charfusion[\mathop]{\bigcup}{\cdot}}

\title{Relatively projective profinie groups}
\author{ Pavel A. Zalesskii}
\date{\today}

\address{P. A. Zalesski\u i\\
Department of Mathematics\\
University of Brasilia\\
70.910 Brasilia DF\\
Brazil}
\email{pz@mat.unb.br}
\begin{document}

\maketitle

\noindent MSC classification: 20E18

\noindent Key-words: profinite groups, prosolvable groups.

\section{Introduction}

It is well-known that
the Schreier theorem does not hold for free profinite groups,
i.e.,
a subgroup of a free profinite group need not be free.
Subgroups of free profinite groups are groups of cohomological dimension $\leq 1$;
they are called projective,
as they satisfy the same universal property as projective modules.

Similarly, the Kurosh Subgroup Theorem
does not hold for free products of profinite groups.
Subgroups of free profinite product are relatively projective that gave motivation to study relatively projective groups in the context of profinite groups and Galois theory (see \cite{Haran, HJ, HJP, HZ, Pop95}).

Following  \cite{HZ} a profinite group  $G$ acting continuously on a profinite space $T$ will be called a {\it profinite pile} (or a pair if one follows \cite[Chapter 7]{DD})  and denoted by $(G,T)$.  We say that $(G,T)$ is projective relatively to its point stabilizers if for given equivariant morphisms $(\varphi,\alpha)$
\begin{equation}\label{EP}
\xymatrix{&(G,T)\ar[d]^{\varphi}\ar@{-->}[dl]_{\psi}\\
             (B,Y)\ar[r]^{\alpha}& (A,X)}
\end{equation}             
             such that $\alpha$ is an epimorphism  injective on point stabilizers that induces a homeomorphism on the spaces of orbits, there exists an equivariant map $\psi:(G,T)\rightarrow (B,Y)$ such that $\alpha\psi=\varphi$. If $G$ is a pro-$\C$ group (where $\C$ is a class of finite groups closed for subgroups, quotients and abelian extensions) then the whole diagram is restricted to pro-$\C$ groups and the notion is the relative $\C$-projectivity. This definition is slightly differs from the definition in \cite{Haran}, \cite{HJ} and \cite{wilkes}; in \cite{wilkes} the definition is more restrictive for profinite groups of big cardinality as the author assumes the existence of a continuous section $\sigma:T/G\longrightarrow T$; in \cite{HJ}  the term is strongly projective; we refer the reader to  \cite[Proposition 5.7]{HZ} for the interrelation of our definition and the one in \cite{HJ}. 

The  objective of this note is to give a homological characterization of relatively prosolvable projective groups in the spirit of Ribes' cohomological characterization of free (amalgamated) products   \cite{R}. Let $\pi$ be a set of primes. We say that a pile $(G,T)$ is $\pi$-projective if the embedding problem (\ref{EP}) with $\Ker(\alpha)$ being an elementary abelian $p$-group is solvable for any $p\in\pi$. If $G$ is prosolvable, $\pi$-projectivity coincides with $\C$-projectivity, where $\C$ is the class of all solvable $\pi$-groups.

\begin{thm}\label{inverse limit of free products}
Let $(G,T)$  be a profinite pile. 
Then the following statements are equivalent:

\begin{enumerate}
\item[(i)]
$(G,T)$ is  $\pi$-projective;

\item[(ii)] For each prime $p\in\pi$ and every  simple $p$-primary $\F_pG$-module $M$ 
the map
$ H_i(G, \F_p[[T]]\otimes M)\to H_i(G, M)$,
induced from $t \mapsto 1$ for every $t \in \dot t\in T/G$,
is surjective for $i=2$ and injective for $i=1$. 
\end{enumerate}
\end{thm}

\begin{cor}\label{solvable} Let $(G,T)$  be a prosolvable pile and $\C$ the class of all finite solvable groups. 
Then  $(G,T)$ is  $\C$-projective if and only if for each prime $p$ and every  simple $p$-primary $\F_pG$-module $M$ 
the map
$ H_i(G, \F_p[[T]]\otimes M)\to H_i(G, M)$,
induced from $t \mapsto 1$ for every $t \in \dot t$,
is surjective for $i=2$ and injective for $i=1$. 

\end{cor}

The structure of the paper is the following. Section 2 is dedicated to auxiliary results on relatively projective groups. Theorem \ref{inverse limit of free products} and Corollary \ref{solvable} are proved in Section 3.  We list open questions about relatively projective groups in Section 4. 

\bigskip

\noindent{\bf Acknowledgements.} 
\begin{enumerate}
\item[(i)]
The idea of the homological characterization of relatively projective groups arose when the author was working jointly with Dan Haran on the  article \cite{HZ};  Sections 2 and 3 of this paper benefited a lot from numerous discussions with him. 

\item[(ii)] During the work on the manuscript the author learned that Theorem \ref{inverse limit of free products} was independetly obtained by Gareth Wilkes \cite{wilkes}  using cohomology of profinite pairs; I am independent to him for sharing his early draft with me.  

\item[(iii)] This paper was written when the author visited the Department of Pure Mathematics and Mathematical Statistics of the University of Cambridge; the  author thanks Henry Wilton and the department for the hospitality. 
\end{enumerate}

\section{Relatively projective groups}

We shall consider here a continuous action on the right of a profinite group $G$ on a profinite space $T$ and use the notation $t^g$ for the translation.

\begin{defi}\label{pile}
A \textbf{pile}
$\bfG = (G,T)$ 
consists of a profinite group $G$,
a profinite space $T$, 
and a continuous action of $G$ on $T$ (from the right).

A pile
$\bfG = (G,T)$
is \textbf{finite} if
both $G$ and $T$ are finite.

A \textbf{morphism} of group piles
$\alpha \colon \bfB=(B,Y) \to \bfA=(A,X)$
consists of
a group homomorphism
$\alpha \colon B \to A$
and a continuous map
$\alpha \colon Y \to X$
such that
$\alpha(y^b) = \alpha(y)^{\alpha(b)}$
for all $y \in Y$ and $b \in B$.

The \textbf{kernel $\Ker \alpha$ of $\alpha$}
is the kernel of the group homomorphism $\alpha \colon B \to A$.

The above morphism $\alpha$ is an \textbf{epimorphism} if
$\alpha(B) = A$, $\alpha(Y) = X$,
and for every $x \in X$ there is $y \in Y$ such that
$\alpha(y) = x$ and $\alpha(B_y) = A_x$.
It is \textbf{rigid}, if
$\alpha$ maps $B_y$ isomorphically onto $A_{\alpha(y)}$,
for all $y \in Y$,
and the induced map of the orbit spaces $Y/B \to X/A$ is a homeomorphism.
\end{defi}

\begin{defi}\label{proj pile}
An \textbf{embedding problem} for a pile $\bfG$ is a pair
\begin{equation}
(\varphi\colon\bfG\to\bfA,\ \alpha\colon\bfB\to\bfA)
\end{equation}
of morphisms of group piles
such that $\alpha$ is an epimorphism.
It is \textbf{finite}, if $\bfB$ is finite and {\bf rigid} if $\alpha$ is rigid. It is a $\C$-embedding problem if $G,B,A$ are pro-$\C$.  It is a $\pi$-embedding problem if $\Ker\alpha$ is an abelian $\pi$-group, where $\pi$ is a set of primes.

A \textbf{solution} of \eqref{EP} is a
morphism $\gamma\colon \bfG\to \bfB$
such that $\alpha\circ\gamma=\varphi$.

A pile $\bfG$ is $\C$-projective (resp. $\pi$-\textbf{projective}),
if every finite rigid $\C$-embedding problem for $\bfG$ (resp. every $\pi$-embedding problem) 
has a solution. 

\end{defi}

\begin{defi}
A group homomorphism (resp. epimorphism) $\gamma:G\rightarrow B$ such that $\alpha\circ\gamma=\varphi$ will be called a weak solution (resp.  proper weak solution) of \eqref{EP}.

 Following  \cite{HJ} we shall say that a pro-$\C$ group $G$ is $\G$-projective (resp.properly $\G$-projective)  if  every finite $\C$-embedding problem has a weak solution (resp. proper weak solution). If  every finite rigid $\pi$-embedding problem has a weak solution (resp. proper weak solution) then $G$ shall be called $\G$-projective (resp.properly $\G$-projective) over $\pi$.
\end{defi}
\medskip


\begin{rem}\label{rel proj}
(a)
Of course if $\bfG$ is projective then $G$ is $\G$-projective.  Moreover, the projectivity of $\bfG$ implies strongly $\G$-projectivity of $G$ in the sense of \cite{HJ} which is called $\G$-projectivity in \cite{HZ}. Conversely, if $G$ is strongly $\G$-projective  and $G_t\neq G_s$ for any pair $t\neq s$ then $\bfG$ is projective (see \cite[Proposition 5.7]{HZ}).

\medskip
(b)
$\bfG$ is $\pi$-projective,  if
every finite rigid embedding problem \eqref{EP} for $\bfG$,
in which $\Ker(\alpha)$ is a minimal normal subgroup of $B$,
has a solution.  So we can always assume that $\Ker(\alpha)$ is elementary abelian.

This follows by induction on $|\Ker(\alpha)|$.
If $\Ker(\alpha)$ is not minimal,
it has a subgroup $C \normal B$,
such that $1 \ne C \ne B$.
Then $\alpha$ factors into rigid epimorphisms $\alpha_1, \alpha_2$
in the diagram below, 
where $\bfB' = (B/C,T/C)$.
By the induction hypothesis
there is a solution  $\gamma_2$ of the embedding problem
$(\varphi, \alpha_2)$,
that gives the embedding problem 
$(\gamma_2, \alpha_1)$,
whose solution  $\gamma$ gives a solution of \eqref{EP}:
\begin{equation*}
\xymatrix{
&& \bfG \ar@{.>}[lld]_{\gamma} \ar@{.>}[ld]^{\gamma_2} \ar[d]^{\varphi}
\\
\bfB \ar[r]_{\alpha_1} & \bfB' \ar[r]_{\alpha_2} & \bfA \rlap{.}
\\
}
\end{equation*} 

\medskip
(c) A profinite group  $G$ is $\G$-projective over $\pi$ if and only if its free profinite product  $G\amalg F_\omega$ is properly $\G$-projective over $\pi$, where $F_\omega$ is free profinite  of countable rank (if we are in the class of pro-$\C$ groups,  then the free product and  $F_\omega$ should be considered in this class).  Indeed,  let $\varphi_\omega:G\amalg F_\omega\longrightarrow A$ be an epimorphism and $\varphi$ its restriction to $G$.   Since any finite embedding problem is properly solvable for $F_\omega$,   a weak solution of $(\alpha,\varphi)$ can be always extended to a proper solution of the induced embedding problem $(\alpha,\varphi_{\omega})$,  and any proper solution of $(\alpha,\varphi_{\omega})$ can be restricted to $G$ producing a weak solution of $(\alpha,  \varphi)$.  
\end{rem}

We finish the section with proposition that will be used in Section 4.

\begin{prop}\label{into free prod} Let $\C$ be a class of finite groups closed for subgroups, quotients and extensions. Let $(G,T)$ be a $\C$-projective pile such that there exists a continuous section $\sigma:T/G\longrightarrow T$. Then $G$ embeds into a free pro-$\C$-product $H=\coprod_{s\in Im(\sigma)} G_s\amalg F$, where $F$ is a free pro-$\C$ group of rank $d(G)$, the minimal number of generators of $G$. Moreover, $\{G_t\mid t\in T\}=\{H_s^h\cap G\mid h\in H, s\in Im(\sigma)\}$.

\end{prop}

\begin{proof} Let $\alpha:H\longrightarrow G$ be an epimorphism that sends $H_s$ on their isomorphic copies in $G$ and $F$ onto $G$. Put $S=\bigcup_{s\in Im(\sigma)} H_s\backslash H$. By \cite[Constriction 4.3]{HZ} $(H,S)$ is a pile and so extending $\alpha$ to $S\rightarrow T$ by $\alpha(H_sh)=s^{\alpha(h)}$ one gets a rigid epimorphism $\alpha:(H,S)\longrightarrow (G,T)$. Since $(G,T)$ is $\C$-projective  the embedding problem $(id,\alpha)$ has a solution which is clearly an embedding.

\end{proof} 

\section{Homological characterization}

Let $G$ be a profinite group and
$\{G_t \suchthat t\in T\}$
be a family of subgroups indexed by a profinite space $T$.
Following \cite[Section 5.2]{R 2017} we say that
it is \textbf{continuous}
if for any open subgroup $U$ of $G$
the subset $\{ t\in T \suchthat G_x \le U\}$ is open.

\begin{lem}[{\cite[Lemma 5.2.1]{R 2017}}]\label{continuous family}
Let $G$ be a profinite group and
let $\{G_t \suchthat t\in T\}$ be a collection of subgroups
indexed by a profinite space $T$.
Then the following conditions are equivalent:

\begin{enumerate}
\item[(a)] $\{G_t \suchthat t\in T\}$ is continuous;

\item[(b)] The set $\hat\G=\{(g,t)\in G\times T \suchthat t\in T, g\in G_t\}$ is closed in $G\times T$;

\item[(c)] The map $\varphi \colon T\to \Subgr(G)$, given by $\varphi(t)=G_t$, is continuous, where $\Subgr(G)$ is endowed with the \'etale topology;

\item[(d)] $\bigcup_{t\in T} G_t$ is closed in $G$.
\end{enumerate}
\end{lem}

If $(G,T)$ is a pile then  denoting by $G_t$ the $G$-stabilizer of $t$,
for every $t \in T$,
we note that
$\G = \{G_t\}_{t \in T}$
is a continuous family of closed subgroups of $G$
(\cite[Lemma 5.2.2]{R 2017})
closed under the conjugation in $G$,
such that
$G_{t^g} = G_t^g$ for all $t \in T$ and $g \in G$. 

Having a continuous family $\{G_t\mid t\in T\}$ of subgroups of a pro-$\C$ group $G$ one can define a free pro-$\C$ product $\coprod_{t\in T} G_t$ (see \cite[Chapter 5]{R 2017}). If $\C$ consists of abelian groups, then the free pro-$\C$ product becomes a profinite direct sum $\bigoplus_{t\in T} G_t$. This also extends to the case of pro-$\C$ modules.

\begin{lem}\label{inf EP}
Let $\bfG=(G,T)$ be a pile  
and let
$\G = \{G_t\}_{t \in T}$ be the family of $G$-stabilizers.
Consider the following rigid embedding problem
\begin{equation}\label{id-EP}
\big(\id \colon \bfG\to \bfG,\alpha\colon \bfB=(B,Y)\to \bfG\big) 
\end{equation}
with  $M=\Ker(\alpha)$ finite minimal abelian normal subgroup (i.e. an elementary abelian  of order $p^n$). Let 
$$\kappa_i \colon 
H_i(G, \F_p[[T]]\otimes M^*)\to H_i(G, M^*)$$
be  the homology maps induced from $T \mapsto 1$,
for $i = 1,2$, where $M^*$ is dual of $M$. 
\begin{itemize}
\item[(a)] 
 $\kappa_2$ is surjective if and only if for any embedding problem \eqref{id-EP} 
 the group homomorphism $\alpha \colon B \to G$
splits
and hence $B \cong M \rtimes G$.
\item[(b)]
    $\kappa_1$ is injective if and only if every split embedding problem
\eqref{id-EP} has a solution.
\end{itemize}
\end{lem}

\begin{proof}
(a)
We just need to show that the extension
$\alpha \colon B \to G$,
as an element $\beta$ of $H^2(G,M)$,
is trivial if and only if $\kappa_2$ is surjective.

Since $\alpha$ is rigid, $\alpha_{|G_y}$ is injective, for all $y\in Y$ which means that 
for every $G_t$
there exists $y\in Y$ and an isomorphism  $\gamma_t: G_t\longrightarrow B_y$ such that
 $\alpha \circ \gamma_t=\id$.
Note that the $G$-orbit
$\dot t$ of $t$ is homeomorphic to $G_t\backslash G$,
hence
$\F_p[[\dot t]] \cong F_p[[G_t\backslash G]]$. Observe also that $Tor_n^{\F_p[[G]]}(M^*,\F_p[[\dot t]])$ is naturally isomorphic to $H_n(G, \F_p[[T]]\otimes M^*)$.
Therefore,
by \cite[Lemma 6.10.10]{RZ} for each $\dot t\in T/G$
one has the following commutative diagram
\begin{equation}\label{Shapiro}
\xymatrix{
H_n(G, \F_p[[\dot t]]\otimes M^*)\ar[dr]^{\epsilon_n}\ar[dd]^{\varphi_n} & \\
            & H_n(G,M^*) \\
H_n(G_t,M^*) \rlap{,} \ar[ur]^{\Cor_t} &
}
\end{equation}
where
$\varphi_n$ is an isomorphism
and $\epsilon_n$ is induced from the augmentation map
$\epsilon \colon \F_p[[\dot t]]\rightarrow \F_p$,
for every $n \ge 0$.

By \cite[Proposition 5.5.4]{R 2017} $\F_p[[T]]=\bigoplus_{\dot t \in T/G}
\F_p[[\dot t]])$ and so $$H_i(G, \F_p[[T]]\otimes M^*)=\oplus_{\dot t \in T/G}
H_i(G, \F_p[[\dot t]]\otimes M^*)$$ by \cite[Theorem 9.1.3 (b)]{R 2017}. 
Thus by \eqref{Shapiro} the surjectivity of
$$\bigoplus_{\dot t\in T/G} H_2(G,\F_p[[\dot t]]\otimes M^*)\to H_2(G,M^*)$$
is equivalent to generation of $H_2(G,M)$ by 
$\Cor_t(H_2(G_t,M^*))$, $t\in T$.
Dually this means that
$\bigcap_{t\in T} \Ker(\Res^2_t)=0$,
where $$\Res_t^2 \colon H^2(G,M)\longrightarrow H^2(G_t,M)$$
is the restriction map
(cf. \cite[Proposition 6.3.6]{R 2017}).

But, for every $t\in T$,
the restriction $\Res_t^2(\beta)$ of $\beta$ to $G_t$ is the $0$-map,
because $\alpha|_{B_t}$ is an isomorphism.
Hence $\bigcap_{t\in T} \Ker(\Res^2_t)=0$ implies $\beta =0$ and conversely $\beta\neq 0$ implies $\bigcap_{t\in T} \Ker(\Res^2_t)\neq 0$ and so non-surjectivity of $\kappa_2$. 

\medskip
(b)
As $B = M \rtimes G$, 
the group $H^1(G,\F_p)$
is isomorphic to the group of continuous monomorphisms of $\gamma:G\longrightarrow B$ with $\alpha\gamma=id$ modulo the conjugation by $M$
 (cf. \cite[Exercise 6.8.2 (b)]{RZ}).
Similarly,
for every $t \in T$,\
$H^1(G_t,M)$
is isomorphic to the group of continuous monomorphisms 
$\gamma:G_t \to M\rtimes G_t$ 
with $\alpha\gamma_t=id$ modulo the conjugation by $M$.
Thus $B_y\leq M\rtimes G_t$ gives an element $\beta_t\in H^1(G_t,M)$. Moreover, we can view $\beta_t$ as a homomorphism $H_1(G_t, M^*)\longrightarrow \F_p$ and its composition with $\varphi_1^{-1}$ from \eqref{Shapiro} gives the continuous homomorphism $$\delta_{\dot t}:H_n(G, \F_p[[\dot t]]\otimes M^*)\longrightarrow \F_p.$$ Besides since $\{G_y\mid y\in Y\}$ is a continuous family (\cite[Lemma 5.2.2]{R 2017}) closed for conjugation,  the maps $$\delta_{\dot t}:H_1(G, \F_p[[\dot t]]\otimes M^*)\longrightarrow \F_p$$ are well defined (i.e. $\beta _t$  and $\beta_{t^g}$ define the same $\delta_{\dot t}$). 

\medskip
Claim. The maps $\delta_{\dot t}$   determine a continuous homomorphism  $$\beta:\bigoplus_{\dot t\in T/G} H_1(G, \F_p[[\dot t]]\otimes M^*) \longrightarrow \F_p.$$ 

\bigskip
Indeed, let $U_i$ be the family of all open normal subgroups of $B$ such that  $M\cap U_i=1$ and $\alpha(U_i)$ acts trivially on $M$. Then there exist decompositions $(B,Y)=\varprojlim_{i\in I} (B/U_i,Y_i)$, $(G,T)=\varprojlim_{i\in I} (G_i,T_i)$  as  inverse limits of finite piles (see \cite[Lemma 5.6.4]{R 2017}), where $G_i=G/\alpha(U_i)$, $T_i=T/\alpha(U_i)$, and we have $B_i=B/U_i=M\rtimes G_i$. Put $T_i=Y_i/M$.  Then we have a splitting homomorphism $\alpha_i:M\rtimes G_i\rightarrow G_i$ and moreover $\alpha_i:(B_i,Y_i)\rightarrow (G_i,T_i)$ is rigid, for each $i$. This means that we have a family of $\beta_{i, t_i}\in H^1({G_i}_{t_i},M)$ which determine homomorphisms $\delta_{i, \dot t_i}:H_1({G_i}, \F_p[\dot t_i]\otimes M^{*})\longrightarrow \F_p$. This defines uniquely the homomorphism $$\beta_i:\oplus_{t_i\in T_i}H_1({G_i}, \F_p[\dot t_i]\otimes M^{*})\rightarrow \F_p.$$ Now since a projective limit commutes with profinite direct sums (cf. \cite[Theorem 5.3.4]{R 2017})
$\beta=\varprojlim_{i\in I}\beta_i$ defines the desired homomorphism $$\beta:\bigoplus_{\dot t\in T/G} H_1({G_i}, \F_p[\dot t_i]\otimes M^{*}) \longrightarrow \F_p.$$

Recall that  $H_1(G, \F_p[[T]]\otimes M^*)=\oplus_{\dot t \in T/G}
H_1(G, \F_p[[\dot t]]\otimes M^*)$. Hence $\kappa_1$ can be rewritten as  $$\kappa_1 \colon \oplus_{\dot t \in T/G}
H_1(G, \F_p[[\dot t]]\otimes M^*)\to H_1(G, M^*)$$
Then  the injectivity of $\kappa_1$ 
allows to extend
$\delta_{\dot t}, \dot t\in T/G$
to $$\beta \colon H_1(G,M^*)\rightarrow \F_p,$$
whose restriction to $G_t$ is $\beta_t$,
for every $t \in T$. 

Conversely, if $\kappa_1$ is not injective and $0\neq k$ is in the kernel of $\kappa_1$, then there exists a choice of $\delta_{\dot t}, \dot t\in T/G$  that determine a homomorphism $$\beta:\bigoplus_{t\in T} H_1(G, \F_p[[\dot t]]\otimes  M^{*}) \longrightarrow \F_p$$ with $\beta(k)\neq 0$. Then there exists $i$ such that for $\beta_i=\oplus_{i}\delta_{i, \dot t_i}$ one has  $\beta_i(k_i)\neq 0$, where $k_i$ is the image of $k$ under the projection $$\bigoplus_{\dot t\in T/G} H_1(G, \F_p[[\dot t]]\otimes M^*)\longrightarrow \oplus_{t_i\in T_i}H_1(G_i, \F_p[[\dot t_i]]\otimes M^*).$$ For each $t_i$ choose a homomorphism $\gamma_{t_i}:{G_i}_{t_i}\longrightarrow B_i$ with $\gamma_{t_i^g}(g_i)=\gamma_{t_i}(g_i)^g$ for all $g_i\in G_{t_i}, g\in G_i$ that represents $\beta_{i, t_i}$  determined   by $\delta_{\dot t_i}$. Define $Y_i=\bigcup_{t_i} B_i/\gamma_{t_i}({G_i}_{t_i})$.  Then  we have a morphism of piles $$\alpha_i:\bfB_i=(B_i, Y_i)\longrightarrow \bfG_i=(G_i,T_i)$$ and the embedding problem $(\alpha_i, \id)$ does not have solution, since $\beta_i$ 
 can not be an element of $H^1(G_i,M^*)=Hom(H_1(G_i,M), \F_p)$. Consider the pull-back 
\begin{equation*}
\xymatrix{
\hat\bfB \ar[r]^\alphahat \ar[d]^{\pi}
& \bfG \ar[d]_{\varphi}
\\
\bfB_i \ar[r]^\alpha & \bfG_i 
}
\end{equation*} 
 Then the embedding problem $(\alphahat,\id)$  does not have a solution, since otherwise $(\alpha,\id)$ does, contradicting  the above.
\end{proof}

\bigskip
\begin{proof}[Proof of Theorem \ref{inverse limit of free products}]
  Suppose for each $p\in \pi$ and every simple $p$-primary module $M$ the map
 $\kappa_2$ is surjective.  Consider a  finite rigid embedding problem
\begin{equation*}
\big(\varphi\colon \bfG\to \bfA,\ \alpha\colon \bfB\to \bfA)\big)
\tag{\ref{EP}}
\end{equation*}
with $M=|\Ker(\alpha)|$ is minimal normal elementary abelian $p$-group for $p\in\pi $ (cf.   Remark~\ref{rel proj}(b)).

Let $\Bhat = B \times_A G$, $\hat Y=Y\times_X T $ be the pull-backs. This gives a pull-back $\hat\bfB=\bfB\times_{\bfA}\bfG$ of piles
and let $\pi \colon \hat\bfB \to \bfB$,
$\alphahat \colon \hat\bfB \to \bfG$
be the coordinate projections.  
Thus we have a commutative diagram
\begin{equation*}
\xymatrix{
\hat\bfB \ar[r]^\alphahat \ar[d]^{\pi}
& \bfG \ar[d]_{\varphi}
\\
\bfB \ar[r]^\alpha & \bfA 
}
\end{equation*}
 By \cite[Lemma 4.8]{HZ}  $\hat\alpha$ is rigid. 

By Lemma \ref{inf EP} (a) if $\kappa_2$ is surjective then  $\hat\alpha$ splits, i.e  there is a morphism
$\psi \colon \bfG \to \bfB$
such that $\alphahat \circ \psi = \id$.  
 
 By Lemma \ref{inf EP} (b) if $\kappa_1$ is injective then the embedding problem $(\id,\hat\alpha)$ has a solution $\hat\gamma$ in which case
$\pi\circ\hat\gamma$ is the required solution. 

\bigskip
Conversely,    if $\bfG$ is not $\pi$-projective then
\eqref{EP} does not have any  solution and so  $(\id,\hat\alpha)$ does not have solution  so by Lemma \ref{inf EP} either $\kappa_2$ is not surjective or   $\kappa_1$ is not injective.
\end{proof}

 \bigskip
\begin{proof}[Proof of Corollary \ref{solvable}] By Remark \ref{rel proj}(b) we may consider only $\C$-embedding problems $(\alpha,\varphi)$ with $\ker(\alpha)$ minimal normal. But  if $\C$ is the class of all finite solvable groups then such $\ker(\alpha)$ is elementary abelian. Thus the result follows from Theorem \ref{inverse limit of free products}. \end{proof}

\begin{cor} Let $(G,T)$ be a  pro-$p$ pile. Then $(G,T)$ is projective if and only if $ H_i(G, \F_p[[T]])\to H_i(G, \F_p)$,
induced from $t \mapsto 1$ for every $t \in \dot t$,
is surjective for $i=2$ and injective for $i=1$.

\end{cor}

\section{open problems}

1. Let $(G,T)$ be a profinite projective pile and $T'=\{t\in T\mid G_t\neq 1\}$. Suppose that $G$ is finitely generated. 

(a) Is $T'/G$ finite?

(b) Is $G_t$ a retract of $G$?   

\bigskip
2. Can one embed a projective profinite pile $(G,T)$ into a free profinite product $H=\coprod_{x\in X} H_x\amalg F$ with $F$ free profinite such that for every $t\in T$  one has
$G_t=H_x^h\cap G$ for some $x\in X$?

\medskip
{\it If there exists a continuous section $T/G\longrightarrow T$ this follows from Proposition \ref{into free prod}. The pro-$p$ case is the subject of \cite{HZ}.}

\bigskip
3. Let $F$ be an abstract free group and $H$ a finitely generated subgroup of $F$. Suppose $(\widehat F,T)$ is a projective pile with $|T|>1$ and  $\overline H=\widehat F_t$ for some $t\in T$. Is $H$ a free factor of $F$?

\medskip
{\it This is the case if $\overline H$ is a free factor of $\widehat F$, see \cite{PP, W, GJ} for 3 different proofs.}

\bigskip
4. Let $F$ be a finitely generated free group and $1\neq g\in F$.  Suppose $(\widehat F,T)$ is a projective pile with $|T|>1$ and $g\in \widehat F_t$ for some $t\in T$. Does $g$ belong to a non-trivial free factor of $F$?

\bigskip
5. Is it true that $(G,T)$ is a projective prosolvable pile if and only if $G$ acts on a profinite tree with trivial edge stabilizers and vertex stabilizers being point stabilizers $G_t$, $t\in T$? 

\medskip
Theorem below states that this is true if $T\longrightarrow T/G$ admits a continuous section.  

\begin{thm}

\begin{enumerate}

\item[(i)]
Let  
$G$ be a prosolvable group  acting on a profinite tree $D$
 with trivial edge stabilizers. Then $(G,V(D))$ is $\C$-projective, where $\C$ is the class of all finite solvable groups. 
 
 \item[(ii)] Let $\C$ be a class of finite groups closed for subgroups, quotients and extensions. Let $(G,T)$ be a projective pro-$\C$ pile. Suppose there exists a continuous section $\sigma: T/G\longrightarrow T$. Then $G$ acts on a pro-$\C$ tree with trivial edge stabilizers and non-trivial vertex stabilizers being stabilizers of points in $T$.
 \end{enumerate}
 \end{thm}

  \begin{proof} (i).
  Suppose $G$ acts on a profinite tree $D$ with trivial edge stabilizers. This means that we have the short exact sequence of $F_p[[G]]$-modules for any prime $p$ 
\begin{equation*}
0 \to \F_p[[E^*,*]] \to \F_p[[V]] \to \F_p \to 0
\end{equation*}
where $\F_p[[E^*]]$ is free $\F_p[[G]]$-module and $V$ is the set of vertices of $D$ (see \cite[Section 2.4]{R 2017}). 

Then, after tensoring it with $\widehat\otimes_{\F_p[[G]]} M$,
we get a long exact sequence in homology
\begin{equation*}
Tor_2^G(M,\F_p[[E^*,*]]) \to Tor_2^G(M,\F_p[[V]]) \to H_2(G,M) \to
Tor_1^G(M,\F_p[[E^*,*]]).
\end{equation*}
Since $F_p[[E^*,*]]$ is a free $\F_p[[G]]$-module,
the left and right terms are $0$.
It remains to observe that $V = \bigdotcup_{\dot v\in V/G} \dot v$
is the disjoint union of orbits,
so $\F_p[[V]] = \bigoplus_{\dot v} \F_p[[\dot v]]$,
and the pro-$p$ direct sum commutes with homology
(\cite[Theorem 9.1.3(b)]{R 2017}). Since $Tor_2^G(M,\F_p[[V]])=H_2(G,\F_p[[V]]\otimes M)$  the result follows from Theorem \ref{inverse limit of free products}.

\medskip

 (ii) By Proposition \ref{into  free prod} $G$ embeds into a free pro-$\C$  product of $H=\coprod_{x\in X} H_x$ such that $\{G_t\mid t\in T\}=\{H_x^h\cap G\mid h\in H, x\in X\}$. Let $S$ be the standard pro-$\C$ tree for  $H=\coprod_{x\in X} H_x$
(\cite[Section 6.3]{R 2017})
By \cite[Lemma 6.3.2(b)]{R 2017}
$H$ acts on $S$,
with trivial edge stabilizers
and vertex stabilizers being conjugates of $H_x$.
Then the restriction of this action to $G$
gives the required action.

\end{proof}

\end{document}